\documentclass[a4paper]{amsart}
\usepackage{amsthm}
\usepackage{amssymb}
\usepackage{enumerate}
\usepackage{verbatim}

\usepackage{tikz-cd} 

\calclayout 

\numberwithin{equation}{section}

\theoremstyle{plain}

\newtheorem{lemma}{Lemma}[section]
\newtheorem{theorem}[lemma]{Theorem}

\newtheorem{corollary}[lemma]{Corollary}

\theoremstyle{definition}
\newtheorem{definition}[lemma]{Definition}
\newtheorem{example}[lemma]{Example}

\newtheorem*{ack}{Acknowledgements}

\theoremstyle{remark} 
\newtheorem{remark}[lemma]{Remark} 
\newtheorem*{claim}{Claim} 

\newcommand{\Aut}{\operatorname{Aut}}

\newcommand{\cosupp}{\operatorname{cosupp}}

\newcommand{\End}{\operatorname{End}}
\newcommand{\Ext}{\operatorname{Ext}}

\newcommand{\Hom}{\operatorname{Hom}}
\newcommand{\id}{\operatorname{id}}

\newcommand{\Ker}{\operatorname{Ker}}

\newcommand{\Mod}{\operatorname{\mathsf{Mod}}}

\newcommand{\Proj}{\operatorname{Proj}}

\newcommand{\sHom}{\underline{\Hom}}

\newcommand{\StMod}{\operatorname{\mathsf{StMod}}}
\newcommand{\stmod}{\operatorname{\mathsf{stmod}}}

\newcommand{\supp}{\operatorname{supp}}
\newcommand{\pisupp}{\pi\text{-}\supp}
\newcommand{\picosupp}{\pi\text{-}\cosupp}

\newcommand{\col}{\colon}
\newcommand{\ges}{{\scriptscriptstyle\geqslant}}

\newcommand{\da}{{\downarrow}}
\newcommand{\lra}{\longrightarrow}
\newcommand{\Iff}{\Longleftrightarrow}

\newcommand{\xra}{\xrightarrow}

\newcommand{\bik}{Benson/Iyengar/Krause}

\def\mcV{\mathcal{V}}

\def\sfC{\mathsf C}

\def\bbG{\mathbb G}
\def\bbP{\mathbb P}
 
\def\bbZ{\mathbb Z}

\newcommand{\fp}{\mathfrak{p}}
\newcommand{\fq}{\mathfrak{q}}

\newcommand{\eps}{\varepsilon}

\newcommand{\gam}{\varGamma}

\def\Ga1{\operatorname{\mathbb G_{a(1)}}\nolimits}

\title[Stratification and $\pi$-cosupport: Finite groups]{Stratification and
  $\pi$-cosupport: Finite groups}

\author[Benson, Iyengar, Krause, and Pevtsova]{Dave Benson, Srikanth B. Iyengar, Henning Krause \\ and Julia Pevtsova}

\address{Dave Benson \\ 
Institute of Mathematics\\ 
University of Aberdeen\\ 
King's College\\ 
Aberdeen AB24 3UE\\ 
Scotland U.K.}

\address{Srikanth B. Iyengar\\ 
Department of Mathematics\\
University of Utah\\ 
Salt Lake City, UT 84112\\ 
U.S.A.}

\address{Henning Krause\\ 
Fakult\"at f\"ur Mathematik\\ 
Universit\"at Bielefeld\\ 
33501 Bielefeld\\ 
Germany.}

\address{Julia Pevtsova\\ 
Department of Mathematics\\ 
University of Washington\\ 
Seattle, WA 98195\\ 
U.S.A.}

\begin{document}

\begin{abstract} 
We introduce the notion of $\pi$-cosupport as a new tool
for the stable module category of a finite group scheme.
In the case of a finite group, we use this to give a new
proof of the classification of tensor ideal localising subcategories.
In a sequel to this paper, we carry out the corresponding
classification for finite group schemes.
\end{abstract}


\date{4th January 2016}

\maketitle

\setcounter{tocdepth}{1}
\tableofcontents

\section*{Introduction}

The theory of support varieties for finitely generated modules over the group algebra  of a finite group began  back in the nineteen eighties with the work of Alperin and Evens~\cite{Alperin/Evens:1982a}, Avrunin and Scott~\cite{Avrunin/Scott:1982a}, Carlson \cite{Carlson:1981b,Carlson:1983a}, among others. An essential ingredient in its development was Carlson's anticipation that for elementary abelian $p$-groups the cohomological definition of support, which takes its roots in Quillen's fundamental  work on mod $p$ group cohomology \cite{Quillen:1971a}, gave the same answer as the ``rank''  definition through restriction to cyclic shifted subgroups.

To deal with infinitely generated modules, Benson, Carlson and Rickard 
\cite{Benson/Carlson/Rickard:1996a} found that they had to introduce 
cyclic shifted subgroups for generic points of subvarieties, defined over 
transcendental extensions of the field of coefficients. Correspondingly, all
the homogeneous prime ideals in the cohomology ring are involved, not 
just the maximal ones. This enabled them to classify the tensor ideal thick 
subcategories of the stable category $\stmod(kG)$ of finitely generated
$kG$-modules \cite{Benson/Carlson/Rickard:1997a}.

It seemed plausible that one should be able to  use modifications of the same  techniques to classify the tensor ideal localising subcategories of the stable category $\StMod(kG)$ of all $kG$-modules, but there were formidable technical obstructions to realising this, and it was not until more than a decade later that this was achieved by the first three authors of this paper \cite{Benson/Iyengar/Krause:2011b}, using a rather different set of ideas than the ones in \cite{Benson/Carlson/Rickard:1997a}. The basic strategy was a series of reductions, via changes of categories, that reduced the problem to that of classifying the localising subcategories of the derived category of differential graded modules over a polynomial ring, where it was solved using methods from commutative algebra. A series of papers \cite{Benson/Iyengar/Krause:2008a, Benson/Iyengar/Krause:2011a, Benson/Iyengar/Krause:2011b, Benson/Iyengar/Krause:2012b}, 
established machinery required to execute this strategy.

In this paper we give an entirely new, and conceptually different, proof of the
classification of the tensor ideal localising subcategories of
$\StMod(kG)$ from \cite{Benson/Iyengar/Krause:2011b}. It is closer in
spirit to the proof of the classification of the tensor ideal thick
subcategories of $\stmod(kG)$ from
\cite{Benson/Carlson/Rickard:1997a}, and rooted essentially in linear
algebra. The crucial new idea is to introduce and study
$\pi$-cosupports for representations.

The inspiration for this comes from two sources. The first is the theory of $\pi$-points  developed by Friedlander and the fourth author \cite{Friedlander/Pevtsova:2007a}  as a suitable generalisation of  cyclic shifted subgroups. Whereas Carlson's original construction only applied to elementary abelian $p$-groups, and required an explicit choice of generators of the group algebra, $\pi$-points allow for a ``rank variety'' description of the cohomological support for any finite group scheme; see \cite[Theorem~3.6]{Friedlander/Pevtsova:2007a} and Section~\ref{se:pi-points} of this paper. Based on this, in \cite{Friedlander/Pevtsova:2007a}   the $\pi$-support of a module over a finite group scheme $G$ defined over a field $k$ is introduced and used to classify the tensor ideal thick subcategories of $\stmod(kG)$ (with an error corrected in a sequel~\cite{Benson/Iyengar/Krause/Pevtsova:2015b} to this paper). But the tensor ideal localising subcategories of $\StMod(kG)$ remained inaccessible by these techniques alone, even for finite groups.

What is required is a $\pi$-version of the notion of cosupports introduced in \cite{Benson/Iyengar/Krause:2012b}. 
The relevance of $\pi$-cosupport is through the following formula for the module of homomorphisms, 
proved in Section~\ref{se:support-and-cosupport}. For any finite group scheme $G$ 
over $k$, and $kG$-modules $M$ and $N$, there is an equality 
\[ 
\picosupp_{G}(\Hom_k(M,N)) = \pisupp_{G}(M) \cap \picosupp_{G}(N). 
\]
To be able to apply this formula to cohomological cosupport developed in [9] one needs to identify the two notions of cosupport. Our strategy for making this identification is to prove that $\pi$-cosupport detects projectivity: a $kG$-module is projective if and only if its $\pi$-cosupport is empty. The desired classification result would then follow from general techniques developed in~\cite{Benson/Iyengar/Krause/Pevtsova:2015b}.

In Section~\ref{se:finite-groups} we prove such a detection theorem for projectivity for \emph{finite groups}.  Besides yielding the desired classification theorem for $\StMod(kG)$ from \cite{Benson/Iyengar/Krause:2011b}, it implies that cohomological support and cosupport
coincide with $\pi$-support and $\pi$-cosupport, respectively.  This remarkable fact is a vast generalisation of Carlson's original anticipation. The different origins of the two notions are reflected in the fact that phenomena that are transparent, or at least easy to detect, for one may be rather opaque and difficult to verify for the other. See Section~\ref{se:applications} for illustrations. 

The corresponding detection theorem for arbitrary finite group schemes has turned out to be more challenging, and is dealt with in the sequel to this paper \cite{Benson/Iyengar/Krause/Pevtsova:2015b}, using a different  approach, where again $\pi$-support and $\pi$-cosupport play a crucial role. This brings us to the second purpose  of this paper: To lay the groundwork for the proof in \cite{Benson/Iyengar/Krause/Pevtsova:2015b}. For this  reason parts of this paper are written in the language of finite group schemes. However, we have attempted to present it in such a way that the reader only interested in finite groups can easily ignore the extra generality.

\section{Finite group schemes}
\label{se:affine}

This section summarises basic property of modules over affine group schemes; for details we refer the reader to Jantzen
\cite{Jantzen:2003a} and Waterhouse \cite{Waterhouse:1979a}.

Let $k$ be a field. An \emph{affine group scheme} $G$ over $k$ is a functor from commutative $k$-algebras to groups, with the property
that, considered as a functor to sets, it is representable as $\Hom_{k\text{\rm -alg}}(R,-)$. The commutative $k$-algebra $R$ has a
comultiplication coming from the multiplicative structure of $G$, and an antipode coming from the inverse. This makes $R$ into a commutative
Hopf algebra called the \emph{coordinate algebra} $k[G]$ of $G$. Conversely, if $k[G]$ is a commutative Hopf algebra over $k$ then
$\Hom_{k\text{\rm -alg}}(k[G],-)$ is an affine group scheme. This work concerns only affine group schemes so henceforth we drop the
qualifier ``affine''.

A group scheme $G$ over $k$ is \emph{finite} if $k[G]$ is finite dimensional as a $k$-vector space. The $k$-linear dual of $k[G]$ is
then a cocommutative Hopf algebra, called the \emph{group algebra} of $G$, and denoted $kG$.

We identify modules over a finite group scheme $G$ with modules over
its group algebra $kG$; this is justified by
\cite[I.8.6]{Jantzen:2003a}. Thus, we will speak of $kG$-modules
(rather than $G$-modules), and write $\Mod kG$ for the category of
$kG$-modules.

\subsection*{Extending the base field} 

Let $G$ be a finite group scheme over a field $k$. If $K$ is a field
extension of $k$, we write $K[G]$ for $K \otimes_k k[G]$, which is a
commutative Hopf algebra over $K$. This defines a group scheme over
$K$ denoted $G_K$, and we have a natural isomorphism
$KG_K\cong K\otimes_k kG$.

For each $kG$-module $M$, we set
\[ M_K:=K \otimes_k M\quad\text{and}\quad M^K:=\Hom_k(K,M),
\] viewed as $KG_K$-modules. When $K$ or $M$ is finite dimensional over
$k$, these are related as follows.

\begin{remark}
\label{re:finite-dimensional} For any $kG$-module $M$, there is a
natural map of $KG_K$-modules
\[ \Hom_k(K,k) \otimes_k M \lra \Hom_k(K,M)\,.
\] This is a bijection when $K$ or $M$ is finite dimensional over
$k$. Then $M^{K}$ is a direct sum of copies of $M_K$ as a $KG_K$-module,
for $\Hom_k(K,k)$ is a direct sum of copies of $K$, as a $K$-vector
space.
\end{remark}

The assignments $M\mapsto M_{K}$ and $M\mapsto M^{K}$ define functors
from $\Mod kG$ to $\Mod KG_K$ that are left and right adjoint,
respectively, to restriction of scalars along the homomorphism of
rings $kG\to KG_K$. The result below collects some basic facts
concerning how these functors interact with tensor products and
homomorphisms. In what follows, the submodule of $G$-invariants of a
$kG$-module $M$ is denoted $M^{G}$; see \cite[I.2.10]{Jantzen:2003a}
for the construction.

\begin{lemma}
\label{le:invariants} Let $G$ be a finite group scheme over $k$ and
$K$ an extension of the field $k$. Let $M$ and $N$ be $kG$-modules. 

There are natural isomorphisms of $KG_K$-modules:
\begin{enumerate}
\item[\rm(i)] $ (M \otimes_k N)_K \cong M_K \otimes_K N_K$.
\item[\rm(ii)] $(M \otimes_k N)^K \cong M_K \otimes_K N^K$ when $M$ is finite
dimensional over $k$.
\item[\rm(iii)] $\Hom_k(M,N)^K \cong \Hom_K(M_K,N^K)$.
\end{enumerate}
There are also natural isomorphisms of $K$-vector spaces:
\begin{enumerate}
\item[\rm(iv)] $\Hom_{kG}(M,N)^K \cong \Hom_{KG_K}(M_K,N^K)$.
\item[\rm(v)] $(M^G)^K \cong (M^K)^{G_K} $.
\end{enumerate}
\end{lemma}

\begin{proof} The isomorphisms in (i), (iii) and (iv) are standard whilst (v) is the special case $M=k$ and $N=M$ of (iv). The
isomorphism in (ii) can be realised as the composition of natural maps
\[
(M\otimes_{k}K) \otimes_{K} \Hom_{k}(K,N) \xra{\ \sim\ } 
M\otimes_{k} \Hom_{k}(K,N) \xra{\ \sim\ }  \Hom_{k}(K,M\otimes_{k}N)
\] 
where the last map is an isomorphism as $M$ is finite
dimensional over $k$.
\end{proof}

\subsection*{Examples of finite group schemes} 
We recall some important classes of finite group schemes relevant to this work.

\begin{example}[Finite groups]
  \label{ex:finite-groups} 
  A finite group $G$ defines a finite group scheme over any field
  $k$. More precisely, the group algebra $kG$ is a finite dimensional
  cocommutative Hopf algebra, and hence its dual is a commutative Hopf
  algebra which defines a group scheme over $k$; it is also denoted
  $G$.  A finite group $E$ is an \emph{elementary abelian $p$-group} if it is
  isomorphic to $(\bbZ/p)^{r}$, for some prime number $p$. The integer
  $r$ is then the \emph{rank} of $E$. Over a field $k$ of
  characteristic $p$, there are isomorphisms of $k$-algebras
\[ 
k[E]\cong k^{\times r}\quad\text{and}\quad kE\cong k[z_1,\dots,z_{r}]/(z_1^p,\dots,z_{r}^p).
\] 
The comultiplication on $kE$ is determined by the map $z_{i}\mapsto z_{i}\otimes 1 + z_{i}\otimes z_{i} + 1\otimes z_{i}$ and the antipode
is determined by the map $z_{i}\mapsto (z_{i}+1)^{p-1}-1$.
\end{example}

\begin{example}[Additive groups]
\label{ex:ga} 
Fix a positive integer $r$ and let $\bbG_{a(r)}$ denote the
finite group scheme whose coordinate algebra is
\[
k[\bbG_{a(r)}]= k[t]/(t^{p^r})
\]
with comultiplication defined by $t\mapsto t\otimes 1 + 1\otimes t$
and antipode $t\mapsto -t$.  There is an isomorphism of $k$-algebras
\[
k\bbG_{a(r)}\cong k[u_0,\dots,u_{r-1}]/(u_0^p,\dots,u_{r-1}^p).
\]
We note that $\bbG_{a(r)}$ is the $r$th Frobenius kernel of the
additive group scheme $\bbG_{a}$ over $k$; see, for instance,
\cite[I.9.4]{Jantzen:2003a}
\end{example}

\begin{example}[Quasi-elementary group schemes]
  \label{ex:quasi-elementary} 
Following Bendel~\cite{Bendel:2001a}, a  group scheme over a field $k$ of positive 
characteristic $p$ is said to be \emph{quasi-elementary} if it is isomorphic to $\bbG_{a(r)} \times (\bbZ/p)^s$. Its group algebra structure is the same as that of an elementary abelian $p$-group.
\end{example}

A finite group scheme $G$ over a field $k$ is \emph{unipotent} if its
group algebra $kG$ is local. Quasi-elementary group schemes are
unipotent. Also, the group scheme over a field of positive
characteristic $p$ defined by a finite $p$-group is unipotent.

\section{$\pi$-points}
\label{se:pi-points} In the rest of this paper $G$ denotes a finite
group scheme defined over a field $k$ of positive characteristic $p$. We
recall the notion of $\pi$-points and basic results about them. The
primary references are the papers of Friedlander and Pevtsova
\cite{Friedlander/Pevtsova:2005a,Friedlander/Pevtsova:2007a}.

\subsection*{$\pi$-points} A $\pi$-\emph{point} of $G$, defined over a
field extension $K$ of $k$, is a morphism of $K$-algebras
\[ 
\alpha\col K[t]/(t^p) \lra KG_K
\] 
that factors through the group algebra of a unipotent abelian
subgroup scheme $C$ of $G_{K}$, and such that $KG_K$ is flat when viewed
as a left (equivalently, as a right) module over $K[t]/(t^{p})$ via
$\alpha$. It should be emphasised that $C$ need not be defined over
$k$; see Examples~\ref{ex:klein}. Restriction along $\alpha$ defines a
functor
\[ 
\alpha^{*}\col \Mod KG_K \lra \Mod K[t]/(t^{p})\,.
\]

The result below extends \cite[Theorem~4.6]{Friedlander/Pevtsova:2007a}, that dealt with $M_{K}$.

\begin{theorem}
\label{th:pi-equivalence} Let $\alpha\col K[t]/(t^p)\to KG_K$ and
$\beta\col L[t]/(t^p)\to LG_L$ be $\pi$-points of $G$. Then the
following conditions are equivalent.
\begin{enumerate}[{\quad\rm(i)}]
\item For any finite dimensional $kG$-module $M$, the module
$\alpha^*(M_K)$ is projective if and only if $\beta^*(M_L)$ is
projective.
\item For any $kG$-module $M$, the module $\alpha^*(M_K)$ is
projective if and only if $\beta^*(M_L)$ is projective.
\item For any finite dimensional $kG$-module $M$, the module
$\alpha^*(M^K)$ is projective if and only if $\beta^*(M^L)$ is
projective.
\item For any $kG$-module $M$, the module $\alpha^*(M^K)$ is
projective if and only if $\beta^*(M^L)$ is projective.
\end{enumerate}
\end{theorem}

\begin{proof} The equivalence of (i) and (ii) is proved in
\cite[Theorem~4.6]{Friedlander/Pevtsova:2007a}. The equivalence of
(iii) and (iv) can be proved in exactly the same way.

(i)$\iff$(iii) Since $M$ is finite dimensional, $M^{K}$ is a direct
sum of copies of $M_{K}$, by Remark~\ref{re:finite-dimensional}. Hence
$\alpha^*(M^K)$ is projective if and only if $\alpha^*(M_K)$ is
projective. The same is true of $\beta^*(M^L)$ and $\beta^*(M_L)$.
\end{proof}

\begin{definition}
\label{de:pi} 
When $\pi$-points $\alpha$ and $\beta$ satisfy the conditions of Theorem~\ref{th:pi-equivalence}, they are said to be
\emph{equivalent}, and denoted $\alpha\sim\beta$.
\end{definition}

For ease of reference, we list some basic properties of $\pi$-points.

\begin{remark}
\label{re:pi-basics}
(1) Let $\alpha\col K[t]/(t^p)\to KG_K$ be a $\pi$-point and $L$ a field
extension of $K$. Then $L\otimes_{K}\alpha\col L[t]/(t^p)\to LG_L$ is a
$\pi$-point and it is easy to verify, say from condition (i) of
Theorem~\ref{th:pi-equivalence}, that $\alpha\sim L \otimes_K \alpha$.

(2) Every $\pi$-point of a subgroup scheme $H$ of $G$ is naturally a 
$\pi$-point of $G$. This follows from the fact that an embedding of group 
schemes always induces a flat map of group algebras.

(3) Every $\pi$-point is equivalent to one that factors through a quasi-elementary 
subgroup scheme over the same field extension; 
see \cite[Proposition~4.2]{Friedlander/Pevtsova:2005a}.
\end{remark}

\subsection*{$\pi$-points and cohomology} The cohomology of $G$ with
coefficients in a $kG$-module $M$ is denoted $H^{*}(G,M)$. It can be
identified with $\Ext_{kG}^{*}(k,M)$.  Recall that $H^{*}(G,k)$ is a
$k$-algebra that is graded-commutative (because $kG$ is a Hopf
algebra) and finitely generated, as was proved by Friedlander and
Suslin \cite[Theorem~1.1]{Friedlander/Suslin:1997a}.

Let $\Proj H^*(G, k)$ denote the set of homogeneous prime ideals
$H^*(G, k)$ that are properly contained in the maximal ideal of
positive degree elements.

Given a $\pi$-point $\alpha\col K[t]/(t^{p})\to KG_K$ we write
$H^{*}(\alpha)$ for the composition of homomorphisms of $k$-algebras.
\[ 
H^{*}(G,k) = \Ext^{*}_{kG}(k,k) \xra{\ K\otimes_{k}-}
\Ext^{*}_{KG_K}(K,K) \lra \Ext^{*}_{K[t]/(t^{p})}(K,K),
\] 
where the second map is induced by restriction. By Frobenius reciprocity and the theorem of Friedlander and Suslin recalled above, $\Ext^{*}_{K[t]/(t^{p})}(K,K)$ is finitely generated as a module over $\Ext^{*}_{KG_K}(K,K)$. Since the former is nonzero, it follows that the radical of the kernel of the map $\Ext^{*}_{KG_K}(K,K) \to \Ext^{*}_{K[t]/(t^{p})}(K,K)$ is a prime ideal different from $\Ext^{\ges 1}_{KG_K}(K,K)$ and hence that the radical of  $\Ker H^{*}(\alpha)$ is a prime ideal in $H^*(G,k)$, different from $H^{\ges 1}(G,k)$.

\begin{remark}
\label{rem:generic-points} 
Fix a point $\fp$ in $\Proj H^{*}(G,k)$. There exists a field $K$ and a $\pi$-point
\[ 
\alpha_\fp\col K[t]/(t^p)\lra KG_K
\] 
such that $\sqrt{\Ker H^{*}(\alpha_{\fp})}=\fp$. In fact, there is such a $K$ that is a finite extension of the degree
zero part of the homogenous residue field at $\fp$; see \cite[Theorem~4.2]{Friedlander/Pevtsova:2007a}.
\end{remark}

It is shown in \cite[Corollary~2.11]{Friedlander/Pevtsova:2007a} that
$\alpha\sim\beta$ if and only if there is an equality
\[ 
\sqrt{\Ker H^{*}(\alpha)} = \sqrt{\Ker H^{*}(\beta)}\,.
\] 
In this way, the equivalence classes of $\pi$-points are in bijection with $\Proj H^*(G,k)$.

\begin{theorem}[{\cite[Theorem~3.6]{Friedlander/Pevtsova:2007a}}] Let
$G$ be a finite group scheme over a field $k$. Taking a $\pi$-point
$\alpha$ to the radical of $\Ker H^{*}(\alpha)$ induces a bijection
between the set of equivalence classes of $\pi$-points of $G$ and
$\Proj H^*(G,k)$.\qed
\end{theorem}

We illustrate these ideas on the Klein four group that will be the
running example in this work.

\begin{example}
\label{ex:klein} 
Let $V=\bbZ/2 \times \bbZ/2$ and $k$ a field of characteristic two. The group algebra $kV$ is isomorphic to $k[x,y]/(x^{2},y^{2})$, where $x+1$ and $y+1$ correspond to the generators of $V$. Let $J=(x,y)$ denote the Jacobson radical of $kV$. It is well-known that $H^{*}(V,k)$ is the symmetric algebra on the $k$-vector space $\Hom_{k}(J/J^{2},k)$; see, for example, \cite[Corollary~3.5.7]{Benson:1998b}. Thus $H^{*}(V,k)$ is a polynomial ring over $k$ in two variables in degree one and $\Proj H^{*}(V,k)\cong \bbP^{1}_{k}$.

The $\pi$-point corresponding to a rational point $[a,b]\in\bbP^{1}_{k}$ (using homogeneous coordinates) is represented
by the map of $k$-algebras
\[ 
k[t]/(t^{p})\lra k[x,y]/(x^{2},y^{2})\quad\text{where $t\mapsto ax + by$.}
\]

More generally, for each closed point $\fp\in \bbP^1_k$ there is some
finite field extension $K$ of $k$ such that $\bbP^1_K$ contains a
rational point $[a',b']$ over $\fp$ (with $\Aut(K/k)$ acting
transitively on the finite set of points over $\fp$).  Then the
$\pi$-point corresponding to $\fp$ is represented by the map of
$K$-algebras
\[ K[t]/(t^{p})\lra K[x,y]/(x^{2},y^{2})\quad\text{where $t\mapsto a'x
+ b'y$.}
\]

Now let $K$ denote the field of rational functions in a variable
$s$. The generic point of $\bbP^{1}_{k}$ then corresponds to the map
of $K$-algebras
\[ K[t]/(t^{p})\lra K[x,y]/(x^{2},y^{2}) \quad\text{where $t\mapsto
x+sy$.}
\]
\end{example}

\section{$\pi$-cosupport and $\pi$-support}
\label{se:support-and-cosupport}

As before, $G$ is a finite group scheme over a field $k$ of positive characteristic $p$. In this section, we introduce a notion of $\pi$-cosupport of a $kG$-module, by analogy with the notion of $\pi$-support introduced in \cite[\S5]{Friedlander/Pevtsova:2007a}. The main result, Theorem~\ref{th:tensor-and-hom-pi}, is a formula that computes the $\pi$-cosupport of a function object, in terms of the $\pi$-support
and $\pi$-cosupport of its component modules.

\begin{definition}
\label{de:cosupport} The \emph{$\pi$-cosupport} of a $kG$-module $M$
is the subset of $\Proj H^{*}(G,k)$ defined by
\[ 
\picosupp_{G}(M) := \{\fp\in\Proj H^*(G,k) \mid \text{$\alpha_\fp^*(\Hom_k(K,M))$ is not projective}\}.
\] 
Here $\alpha_\fp\col K[t]/(t^p)\to KG_K$ denotes a representative of the equivalence class of $\pi$-points corresponding to $\fp$; see
Remark~\ref{rem:generic-points}. The definition is modelled on that of the \emph{$\pi$-support} of $M$, introduced in \cite{Friedlander/Pevtsova:2007a} as the subset
\[ 
\pisupp_{G}(M) := \{\fp\in\Proj H^*(G,k) \mid \text{$\alpha_\fp^*(K\otimes_k M)$ is not projective}\}.
\] 
This is denoted $\Pi(G)_{M}$ in \cite{Friedlander/Pevtsova:2007a}; our notation is closer to the one used in \cite{\bik:2008a} for cohomological support.
\end{definition}

\subsection*{Projectivity} For later use, we record the well-known
property that the module of homomorphisms preserves and detects
projectivity.

\begin{lemma}
\label{le:end-projectivity} Let $M$ and $N$ be $kG$-modules.
\begin{enumerate}[{\quad\rm(i)}]
\item If $M$ or $N$ is projective, then so is $\Hom_{k}(M,N)$.
\item $M$ is projective if and only if $\End_{k}(M)$ is projective.
\end{enumerate}
\end{lemma}

\begin{proof} We repeatedly use the fact that a $kG$-module is
projective if and only if it is injective; see, for example,
\cite[Lemma~I.3.18]{Jantzen:2003a}.

(i) The functor $\Hom_{k}(M,-)$ takes injective $kG$-modules to
injective $kG$-modules because it is right adjoint to an exact
functor.  Thus, when $N$ injective, so is $\Hom_{k}(M,N)$. The same
conclusion follows also from the projectivity of $M$ because
$\Hom_{k}(-,N)$ takes projective modules to injective modules, as
follows from the natural isomorphisms:
 \begin{align*} 
\Hom_{kG}(-,\Hom_{k}(M,N)) 
	&\cong \Hom_{kG}(-\otimes_{k}M,N) \\ 
	&\cong \Hom_{kG}(M\otimes_{k}-,N) \\
	&\cong \Hom_{kG}(M, \Hom_{k}(-,N))\,.
 \end{align*} 
The first and the last isomorphisms are by adjunction, and the one in the middle holds because $kG$ is cocommutative.
 
(ii) When $M$ is projective, so is $\End_k(M)$, by (i). For the converse, observe that when $\End_{k}(M)$ is projective, so is $\End_{k}(M)\otimes_{k}M$, since $-\otimes_k M$ preserves projectivity being the left adjoint of an exact functor. It remains to note that $M$ is a direct summand of $\End_{k}(M)\otimes_{k}M$, because the composition of the homomorphisms
\begin{gather*} 
M\xra{\ \nu\ } \End_{k}(M)\otimes_{k}M \xra{\ \eps\ }
M,\quad\text{where} \\ \nu(m)=\id_{M}\otimes m\quad\text{and}\quad \eps(f\otimes m) = f(m),
\end{gather*} 
of $kG$-modules equals the identity on $M$.
\end{proof}

We now work towards Theorem~\ref{th:tensor-and-hom-pi} that gives a
formula for the cosupport of a function object, and the support of a
tensor product. These are useful  for studying modules over
finite group schemes, as will become clear in
Section~\ref{se:applications}; see also
\cite{Benson/Iyengar/Krause/Pevtsova:2015b}.

\subsection*{Function objects and tensor products} 
The proof of Theorem~\ref{th:tensor-and-hom-pi} is complicated by 
the fact that, in general, a $\pi$-point $\alpha\col K[t]/(t^p)\to KG_K$ 
does not preserve Hopf structures, so restriction along $\alpha$ does 
not commute with taking tensor products, or the module of homomorphisms. 
To deal with this situation, we adapt an idea from the proof of 
\cite[Lemma~3.9]{Friedlander/Pevtsova:2005a}---see also \cite[Lemma~6.4]{Carlson:1983a} 
and \cite[Lemma~6.4]{Bendel/Friedlander/Suslin:1997b}---where the 
equivalence of (i) and (iii) in the following result is proved. The hypothesis 
on the algebra $A$ is motivated by Remark~\ref{re:pi-basics}(3) and 
Example~\ref{ex:quasi-elementary}.

\begin{lemma}
\label{le:pi-proj} 
Let $K$ be a field of positive characteristic $p$ and $A$ a cocommutative 
Hopf $K$-algebra that is isomorphic as an algebra to $K[t_{1},\dots,t_{r}]/(t_{1}^{p},\dots t_{r}^{p})$.  
Let
\[ 
\alpha\col K[t]/(t^p) \lra A
\] 
be a flat homomorphism of $K$-algebras. For any $A$-modules $M$ and $N$, 
the following conditions are equivalent:
\begin{enumerate}[\quad\rm(i)]
\item $\alpha^*(M \otimes_K N)$ is projective.
\item $\alpha^*(\Hom_K(M,N))$ is projective.
\item $\alpha^*(M)$ or $\alpha^*(N)$ is projective.
\end{enumerate}
\end{lemma}

\begin{proof} 
As noted before, (i) $\iff$ (iii) is \cite[Lemma~3.9]{Friedlander/Pevtsova:2005a}; 
the hypotheses of \textit{op.~cit.} includes that $M$ and $N$ are finite dimensional, 
but that is not used in the proof. We employ a similar argument to verify that (ii) 
and (iii) are equivalent.

Let $\sigma\col A\to A$ be the antipode of $A$, $\Delta\col A\to A\otimes_{K}A$ 
its comultiplication, and set $I=\Ker(A\to K)$, the augmentation ideal of $A$.  
Identifying $t$ with its image in $A$, one has
\[ 
(1\otimes\sigma)\Delta(t) = t\otimes 1 - 1\otimes t + w \quad\text{with $w\in I\otimes_{K}I$;}
\] 
see \cite[I.2.4]{Jantzen:2003a}. Recall that the action of $a\in A$ on $\Hom_{K}(M,N)$ 
is given by multiplication with $(1\otimes\sigma)\Delta(a)$, so that for $f\in \Hom_{K}(M,N)$ 
and $m\in M$ one has
\[ 
(a\cdot f)(m) = \sum a'f(\sigma(a'')m) \quad\text{where $\Delta(a)=\sum a'\otimes a''$.}
\] 
Given a module over $A\otimes_{K}A$, we consider two $K[t]/(t^{p})$-structures on it: 
One where $t$ acts via multiplication with $(1\otimes\sigma)\Delta(t)$ and another where 
it acts via multiplication with $t\otimes 1 - 1\otimes t$. We claim that these two $K[t]/(t^{p})$-modules 
are both projective or both not projective. This follows from a repeated use of 
\cite[Proposition~2.2]{Friedlander/Pevtsova:2005a} because $w$ can be represented as 
a sum of products of nilpotent elements of $A\otimes_{K}A$, and each nilpotent element 
$x$ of $A\otimes_{K}A$ satisfies $x^p=0$.

We may thus assume that $t$ acts on $\Hom_{K}(M,N)$ via $t\otimes 1 - 1\otimes t$. 
There is then an isomorphism of $K[t]/(t^{p})$-modules
\[ 
\alpha^{*}(\Hom_{K}(M,N)) \cong \Hom_K(\alpha^{*}(M),\alpha^{*}(N))\,,
\] 
where the action of $K[t]/(t^{p})$ on the right hand side is the one obtained by viewing 
it as a Hopf $K$-algebra with comultiplication defined by $t\mapsto t\otimes 1 + 1\otimes t$ 
and antipode $t\mapsto -t$. It remains to observe that for any $K[t]/(t^{p})$-modules $U,V$,
the module $\Hom_K(U,V)$ is projective if and only if one of $U$ or $V$ is projective.

Indeed, if $U$ or $V$ is projective, so is $\Hom_{K}(U,V)$, by Lemma~\ref{le:end-projectivity}. 
As to the converse, every $K[t]/(t^{p})$-module is a direct sum of cyclic modules, so when $U$ 
and $V$ are not projective, they must have direct summands isomorphic to $K[t]/(t^{u})$ and 
$K[t]/(t^{v})$, respectively, for some $1\le u,v <p$. Then $\Hom_{K}(K[t]/(t^{u}),K[t]/(t^{v}))$ is 
a direct summand of $\Hom_{K}(U,V)$. The former cannot be projective as a $K[t]/(t^{p})$-module 
because its dimension as a $K$-vector space is $uv$, while the dimension of any projective 
module must be divisible by $p$: projective $K[t]/(t^{p})$-modules are free.
\end{proof}

The first part of the result below is \cite[Proposition~5.2]{Friedlander/Pevtsova:2007a}.

\begin{theorem}
\label{th:tensor-and-hom-pi} Let $M$ and $N$ be $kG$-modules. Then
there are equalities
\begin{enumerate}[{\quad\rm(i)}]
\item $\pisupp_{G}(M \otimes_k N) = \pisupp_{G}(M) \cap
\pisupp_{G}(N)$,
\item $\picosupp_{G}(\Hom_k(M,N)) = \pisupp_{G}(M) \cap
\picosupp_{G}(N)$.
\end{enumerate}
\end{theorem}

\begin{proof} 
We prove part (ii). Part (i) can be proved in the same fashion; see 
\cite[Proposition~5.2]{Friedlander/Pevtsova:2007a}.

Fix a $\pi$-point $\alpha\colon K[t]/(t^{p})\to KG_K$. By Remark~\ref{re:pi-basics}(3), 
we can assume $\alpha$ factors as $K[t]/(t^{p})\to KC\to KG_K$, where $C$ is a 
quasi-elementary subgroup scheme of $G_{K}$. As noted in Lemma~\ref{le:invariants}(iii), 
there is an isomorphism of $KG_K$-modules
\[ 
\Hom_k(M,N)^K \cong \Hom_K(M_{K},N^{K})\,.
\] 
We may restrict a $KG_K$ module to $K[t]/(t^{p})$ by first restricting to $KC$, 
and $\Hom_{k}(-,-)$ commutes with this operation. Thus the desired result follows 
from the equivalence of (ii) and (iii) in Lemma~\ref{le:pi-proj}, applied to the map 
$K[t]/(t^{p})\to KC$, keeping in mind the structure of $KC$; see Example~\ref{ex:quasi-elementary}.
\end{proof}

\subsection*{Basic computations} Next we record, for later use, some elementary 
observations concerning support and cosupport. The converse of (i) in the result 
below also holds. For finite groups this is proved in Theorem~\ref{th:finite-groups} 
below, and for finite group schemes this is one of the main results of \cite{Benson/Iyengar/Krause/Pevtsova:2015b}.

\begin{lemma}
\label{le:cosupp-k} 
Let $M$ be a $kG$-module.
\begin{enumerate}[{\quad\rm(i)}]
\item $\pisupp_{G}(M)=\varnothing=\picosupp_{G}(M)$ when $M$ is
projective.
\item $\picosupp_{G}(M)=\pisupp_{G}(M)$ when $M$ is finite
dimensional.
\item $\pisupp_{G}(k)=\Proj H^{*}(G,k) =\picosupp_{G}(k)$.
\end{enumerate}
\end{lemma}

\begin{proof} Part (ii) is immediate from definitions, given
Remark~\ref{re:finite-dimensional}. For the rest, fix a $\pi$-point
$\alpha\col K[t]/(t^{p})\to KG_K$.

 (i) When $M$ is projective, so are the $KG_K$-modules $K\otimes_{k}M$
and $\Hom_k(K,M)$, and restriction along $\alpha$ preserves
projectivity, as $\alpha$ is a flat map. This justifies (i).

(iii) Evidently, $k_{K}$ equals $K$ and $\alpha^*(K)$ is
non-projective. Since $\alpha$ was arbitrary, one deduces that
$\pisupp_{G}(k)$ is all of $\Proj H^{*}(G,k)$. That this also equals
the $\pi$-cosupport of $k$ now follows from (ii).
\end{proof}

\begin{corollary}
\label{co:cosuppM*} If $M$ is a $kG$-module, then there is an equality
\[ \picosupp_{G}(\Hom_{k}(M,k))=\pisupp_{G}(M).
\]
\end{corollary}

\begin{proof} 
This follows from Theorem~\ref{th:tensor-and-hom-pi}\,(ii) and Lemma~\ref{le:cosupp-k}\,(iii).
\end{proof}

The equality of $\pi$-support and $\pi$-cosupport for finite dimensional modules,  which holds by Lemma~\ref{le:cosupp-k}(ii), may fail for infinite dimensional modules. We describe an example over the Klein  four group; see Example~\ref{ex:T(I)} for a general construction.

\begin{example}
\label{ex:klein2} Let $V$ be the Klein four group $\bbZ/2 \times
\bbZ/2$ and $k$ a field of characteristic two. The $\pi$-points of
$kV$ were described in Example~\ref{ex:klein}. We keep the notation
from there. Let $M$ be the infinite dimensional $kG$-module with basis
elements $u_i$ and $v_i$ for $i\ge 0$ and $kG$-action defined by
\[ 
xu_i=v_i,\qquad yu_i = v_{i-1}, \qquad xv_i=0,\qquad yv_i=0
\] 
where $v_{-1}$ is interpreted as the zero vector. A diagram for this module is as follows:
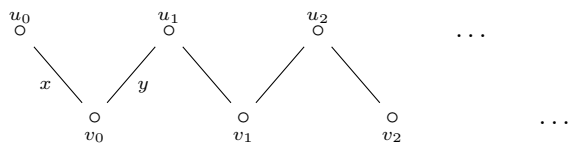
\[
\begin{tikzcd}[column sep=small] 
\overset{u_0}{\circ}
\arrow[dash]{dr}[swap]{x} && \overset{u_1}{\circ}\arrow[dash]{dl}{y}
\arrow[dash]{dr} && \overset{u_2}{\circ} \arrow[dash]{dl}
\arrow[dash]{dr}&&\dots \\
&\underset{v_0}{\circ}&&\underset{v_1}{\circ}&&
\underset{v_2}{\circ}&&\dots
\end{tikzcd}
\]
\begin{claim} 
The $\pi$-support of $M$ is the closed point $\{[0,1]\}$ of $\bbP^{1}_{k}$ whilst its $\pi$-cosupport contains also the generic
point.
\end{claim}

Given a finite field extension $K$ of $k$, it is not hard to verify
that for any rational point $[a,b]\in\bbP^{1}_{K}$ the image of
multiplication by $ax+by$ on $M_K$ is its socle, that is spanned by
the elements $\{v_{i}\}_{i\ges 0}$. For $[a,b] \ne [0,1]$, this is
also the kernel of $ax+by$, whilst for $[a,b]=[0,1]$ it contains, in
addition, the element $u_{0}$. In view of
Remark~\ref{re:finite-dimensional}, this justifies the assertions
about the closed points of $\bbP^1_k$.

For the generic point let $K$ denote the field of rational functions
in a variable $s$. It is again easy to check that the kernel of $x+sy$
and the image of multiplication by $x+sy$ on $M_{K}$ are equal to its
socle. So the generic point is not in $\pisupp_{V}(M)$.

For cosupport, consider the $k$-linear map $f\colon K \to M$ defined
as follows. Given a rational function $\phi(s)$, its partial fraction
expansion consists of a polynomial $\psi(s)$ plus the negative parts
of the Laurent expansions at the poles. If
$\psi(s)=\alpha_0+\alpha_1s+\cdots$ we define $f(\phi)$ to be
$\alpha_0u_0+\alpha_1u_1+\cdots$.  By definition,
$(x+sy)(f)(a)=xf(a)+yf(sa)$; using this it is easy to calculate that
$f$ is in the kernel of $x+sy$. On the other hand, any function in the
image of $x+sy$ lands in the socle of $M$. It follows that the kernel
of $x+sy$ is strictly larger than the image, and so the generic point
is in the $\pi$-cosupport of $M$.
\end{example}

\section{Finite groups} 
\label{se:finite-groups}
The focus of the rest of the paper is on finite groups. In this section 
we prove that a module over a finite group is projective if and only if 
it has empty $\pi$-cosupport. The key ingredient is a version of Dade's 
lemma for elementary abelian groups. This in turn is based on an analogue 
of the Kronecker quiver lemma~\cite[Lemma~4.1]{Benson/Carlson/Rickard:1996a}.

\begin{lemma}
\label{le:kronecker} 
Let $k$ be an algebraically closed field, $\ell$ a non-trivial extension field of $k$, and let $V,W$ be $k$-vector spaces. If there exist $k$-linear maps $f,g\colon V\to W$ with the property that for every pair of scalars $\lambda$ and $\mu$ in $\ell$, not both zero, the linear map $\lambda f + \mu g \col \Hom_k(\ell,V) \to \Hom_k(\ell,W)$ is an isomorphism, then $V=0=W$.
\end{lemma}

\begin{remark} 
Note that $\lambda f + \mu g$ here really means $\Hom_{k}(\lambda,f)+\Hom_{k}(\mu,g)$ since this is the way $\ell$ acts on
homomorphisms.
\end{remark}

\begin{proof} 
Since $k$ is algebraically closed, we may as well assume that $\ell=k(s)$, a simple transcendental extension of $k$, since further
extending the field only strengthens the hypothesis without changing the conclusion.

Use $g$ to identify $V$ and $W$. Then $f$ is a $k$-endomorphism of $V$ with the property that for all $\mu\in \ell$ the endomorphism $\Hom_{k}(1,f) + \Hom_{k}(\mu, \id)$ of $\Hom_k(\ell,V)$ is invertible. The action of $f$ makes $V$ into a $k[t]$-module, with $t$ acting as $f$ does. Since
$f+\mu.\id$ is invertible for all $\mu\in k$, and $k$ is algebraically closed, $V$ is a $k(t)$-module.

Consider the homomorphism $k(t)\otimes_{k}k(s)\to k(t)$ of rings that
is induced by the assignment $p(t)\otimes q(s)\mapsto p(t)q(t)$. It is
not hard to verify that its kernel is the ideal generated by $t-s$, so
there is an exact sequence of $k(t)\otimes_{k}k(s)$-modules
\[ 
0\lra k(t)\otimes_{k}k(s) \xra{\ t - s\ } k(t)\otimes_{k}k(s) \lra k(t) \lra 0
\] 
Applying $\Hom_{k(t)}(-,k(t))$ and using adjunction yields an exact sequence
\[ 
0\lra k(t) \lra \Hom_{k}(k(s),k(t)) \xra{\ t - s\ } \Hom_{k}(k(s),k(t)) \lra 0\,.
\] 
Thus $t-s$ is not invertible on $\Hom_{k}(k(s),k(t))$.  If $V$ is nonzero then it has $k(t)$ as a summand as a $k(t)$-module, so
that $t-s$, that is to say, $\Hom_{k}(1,t)+\Hom_{k}(-s, \id)$, is not invertible. This contradicts the hypothesis.
\end{proof}

\subsection*{Dade's lemma for cosupport} 
The result below, and its proof, are modifications of \cite[Theorem~5.2]{Benson/Carlson/Rickard:1996a}. Given an $k$-algebra $R$ and an extension field $K$ of $k$, we write $R_{K}$ for the $K$-algebra $K\otimes_{k}R$ and
$M^{K}$ for the $R_{K}$-module $\Hom_{k}(K,M)$.

\begin{theorem}
\label{th:DadeHom} Let $k$ be a field of positive characteristic $p$
and set
\[ 
R=k[t_1,\dots,t_r]/(t_1^p,\dots,t_r^p).
\] 
Let $K$ be an algebraically closed field of transcendence degree at
least $r-1$ over $k$. Then an $R$-module $M$ is projective if and only
if for all flat maps $\alpha\col K[t]/(t^r)\to R_K$ the module
$\alpha^*(M^K)$ is projective.
\end{theorem}

\begin{proof} If $M$ is projective as an $R$-module, then $M^K$ is
projective as an $R_{K}$-module, and because $\alpha$ is flat, it
follows that $\alpha^*(M^K)$ is projective.

Assume that $\alpha^*(M^K)$ is projective for all $\alpha$ as in the
statement of the theorem. We verify that $M$ is projective as an
$R$-module by induction on $r$, the case $r=1$ being trivial. Assume
$r\ge 2$ and that the theorem is true with $r$ replaced by $r-1$.

It is easy to verify that the hypothesis and the conclusion of the
result are unchanged if we pass from $k$ to any extension field in
$K$.  In particular, replacing $k$ by its algebraic closure in $K$ we
may assume that $k$ is itself algebraically closed. The plan is to use
Lemma~\ref{le:kronecker} to prove that $\Ext^{1}_{R}(k,M)=0$. Since
$R$ is Artinian it would then follow that $M$ is injective, and hence
also projective, because $R$ is a self-injective algebra.

We first note that for any extension field $\ell$ of $k$, tensoring with
$\ell$ gives a one-to-one map $\Ext^{*}_{R}(k,k)\to
\Ext^{*}_{R_{\ell}}(\ell,\ell)$, which we view as an inclusion, and a natural
isomorphism of $\ell$-vector spaces
\[ 
\Ext^i_R(k,M)^\ell \cong \Ext^i_{R_\ell}(\ell,M^\ell)\quad\text{for $i\in\bbZ$.}
\] 
These remarks will be used repeatedly in the argument below. They imply, for example, that $\Ext^{i}_{R}(k,M)=0$ if and only if
$\Ext^i_{R_\ell}(\ell,M^\ell)=0$.

Let $\beta\col \Ext^1_R(k,k)\to \Ext^2_R(k,k)$ be the Bockstein map; see \cite[\S4.3]{Benson:1998c} or, for a slightly different approach,
\cite[I.4.22]{Jantzen:2003a}. Recall that this map is semilinear through the Frobenius map, in the sense that
\[ 
\beta(\lambda\eps)=\lambda^p\beta(\eps)\quad\text{for $\eps\in \Ext^{1}_{R}(k,k)$ and $\lambda\in k$}.
\]

Fix an extension field $\ell$ of $k$ in $K$ that is algebraically closed
and of transcendence degree $1$.  Choose linearly independent elements
$\eps$ and $\gamma$ of $\Ext^1_R(k,k)$.  The elements $\beta(\eps)$
and $\beta(\gamma)$ of $\Ext^{2}_{R}(k,k)$ induce $k$-linear maps
\[ 
f,g\col \Ext^1_{R}(k,M) \lra \Ext^3_{R}(k,M)\,.
\] 
Let $\lambda$ and $\mu$ be elements in $\ell$, not both zero, and
consider the element
\begin{equation}
\label{eq:deg1elt} \lambda^{1/p}\eps + \mu^{1/p}\gamma \in
\Ext^1_{R_\ell}(\ell,\ell)\cong \Hom_\ell(J/J^2,\ell)\,,
\end{equation} where $J$ is the radical of the ring $R_{\ell}$. It
defines a linear subspace of codimension one in the $\ell$-linear span of
$t_1,\dots,t_r$ in $R_{\ell}$. Let $S$ be the subalgebra of $R_{\ell}$
generated by this subspace and view $M^{\ell}$ as an $S$-module, by restriction of
scalars.

\begin{claim} 
As an $S$-module, $M^{\ell}$ is projective.
\end{claim}

Indeed, $S$ is isomorphic to
$\ell[z_{1},\dots,z_{r-1}]/(z_{1}^{p},\dots,z_{r-1}^{p})$, as an
$\ell$-algebra. It is not hard to verify that the hypotheses of the
theorem apply to the $S$-module $M^{\ell}$ and the extension field
$\ell\subset K$. Since the transcendence degree of $K$ over $\ell$ is $r-2$,
the induction hypothesis yields that the $S$-module $M^{\ell}$ is
projective, as claimed.\medskip

We give $R_\ell$ the structure of a Hopf algebra by making the generators $t_i$ primitive: that is, comultiplication 
is determined by the map $t_i \mapsto  t_i \otimes 1 + 1 \otimes t_i$ and the antipode is determined by the map 
$t_i \mapsto -t_i$. Note that $R_{\ell}\otimes_{S}\ell$ has a natural
structure of an $R_{\ell}$-module, with action induced from the left hand
factor.

\begin{claim} 
  The $R_{\ell}$-module $(R_{\ell}\otimes_{S}\ell)\otimes_{\ell}M^{\ell}$, with the
  diagonal action, is projective.
\end{claim}

Indeed, it is not hard to see the comultiplication on $R_\ell$ induces one on $S$, so the latter is a sub Hopf-algebra of the former. As $M^{\ell}$
is projective as an $S$-module, by the previous claim, so is $\Hom_{\ell}(M^{\ell},N)$ for any $S$-module $N$; see Lemma~\ref{le:end-projectivity}. Since projective $S$-modules are injective, the desired claim is then a consequence of the following standard isomorphisms of functors
\begin{align*} 
\Hom_{R_{\ell}}((R_{\ell}\otimes_{S}\ell)\otimes_{\ell}M^{\ell},-)
&\cong \Hom_{R_{\ell}}(R_{\ell}\otimes_{S}\ell,\Hom_{\ell}(M^{\ell},-)) \\ &\cong
\Hom_{S}(\ell,\Hom_{\ell}(M^{\ell},-))
\end{align*} 
on the category of $R_{\ell}$-modules.\medskip

The Bockstein of the element \eqref{eq:deg1elt} is
\[ 
\lambda \beta(\eps)+\mu\beta(\gamma) \in \Ext^2_{R_\ell}(\ell,\ell),
\] 
and is represented by an exact sequence of the form
\begin{equation}
\label{eq:ext2'} 
0 \lra \ell \lra R_{\ell} \otimes_{S} \ell \lra R_{\ell} \otimes_{S} \ell \lra \ell \lra 0\,.
\end{equation} 
For any $R_{\ell}$-module $N$ the Hopf algebra structure on $R_{\ell}$
gives a map of $\ell$-algebras
$\Ext^{*}_{R_{\ell}}(\ell,\ell) \to \Ext^{*}_{R_{\ell}}(N,N)$ such
that the two actions of $\Ext^{*}_{R_{\ell}}(\ell,\ell)$ on
$\Ext^{*}_{R_{\ell}}(\ell,N)$ coincide, up to the usual
sign~\cite[Corollary~3.2.2]{Benson:1998b}. What this entails is that
the map
\[ 
\lambda f + \mu g \col \Ext^1_{R_\ell}(\ell,M^\ell) \lra \Ext^3_{R_\ell}(\ell,M^\ell)
\] 
may be described as splicing with the extension
\[ 
0 \lra M^{\ell} \lra (R_{\ell}\otimes_{S}\ell)\otimes_{\ell} M^{\ell} \lra
(R_{\ell}\otimes_{S}\ell)\otimes_{\ell} M^{\ell} \lra M^{\ell} \lra
0\, ,
\] 
which is obtained from the exact sequence \eqref{eq:ext2'} by applying
$-\otimes_{\ell} M^{\ell}$. By the preceding claim, the modules in the
middle are projective, and so the element
$\lambda \beta(\eps)+ \mu \beta(\gamma)$ induces a stable isomorphism
\[ 
\Omega^2(M^{\ell}) \xra{\ \sim\ } M^{\ell}\,.
\] 
It follows that $\lambda f + \mu g$ is an isomorphism for all $\lambda, \mu$ in $l$ not both zero. Thus Lemma~\ref{le:kronecker}
applies and yields $\Ext^1_R(k,M)=0$ as desired.
\end{proof}

\subsection*{Support and cosupport detect projectivity} 

The theorem below is the main result of this work. Several
consequences are discussed in the subsequent sections.

\begin{theorem}
\label{th:finite-groups} Let $k$ be a field and $G$ a finite
group. For any $kG$-module $M$, the following conditions are
equivalent.
\begin{enumerate}[\quad\rm(i)]
\item $M$ is projective.
\item $\pisupp_{G}(M)=\varnothing$
\item $\picosupp_{G}(M)=\varnothing$.
\end{enumerate}
\end{theorem}

\begin{proof} We may assume that the characteristic of $k$, say $p$,
divides the order of $G$. The implications (i)$\implies$(ii) and
(i)$\implies$(iii) are by Lemma~\ref{le:cosupp-k}.

(iii)$\implies$(i) Let $E$ be an elementary abelian $p$-subgroup of
$G$ and let $M\da_{E}$ denote $M$ viewed as a $kE$-module. The
hypothesis implies $\picosupp_{E}(M\da_{E})=\varnothing$, by
Remark~\ref{re:pi-basics}(2), and then it follows from
Theorem~\ref{th:DadeHom} that $M\da_{E}$ is projective.  Chouinard's
theorem~\cite[Theorem~1]{Chouinard:1976a} thus implies that $M$ is
projective.

(ii)$\implies$(i) When $\pisupp_{G}(M)=\varnothing$, it follows from
Theorem~\ref{th:tensor-and-hom-pi} that
\[ \picosupp_{G}\End_{k}(M)=\varnothing\,.
\] The already settled implication (iii)$\implies$(i) now yields that
$\End_{k}(M)$ is projective. Hence $M$ is projective, by
Lemma~\ref{le:end-projectivity}.
\end{proof}

\section{Cohomological support and cosupport}
\label{se:applications} 

The final part of this paper is devoted to applications of Theorem~\ref{th:finite-groups}.  We proceed in several steps and derive global results about the module category of a  finite group from local properties, including a comparison of $\pi$-support and 
$\pi$-cosupport with cohomological support and cosupport. In the next section we consider the classification of  thick and localising subcategories of the stable module category.

From now on $G$ denotes a finite group, $k$ a field of positive
characteristic dividing the order of $G$. Let $\StMod(kG)$ be the
stable module category of all (meaning, also infinite dimensional)
$kG$-modules modulo the projectives; see, for example,
\cite[\S2.1]{Benson:1998b}. This is not an abelian category; rather,
it has the structure of a compactly generated tensor triangulated
category and comes equipped with a natural action of the cohomology
ring $H^*(G,k)$; see \cite[Section
10]{Benson/Iyengar/Krause:2008a}. This yields the notion of
cohomological support and cosupport developed in
\cite{\bik:2008a,\bik:2012b}. More precisely, for each homogeneous
prime ideal $\fp$ of $H^*(G,k)$ that is different from the maximal
ideal of positive degree elements, there is a distinguished object
$\gam_\fp k$. Using this one defines for each $kG$-module $M$ its
\emph{cohomological support}
\begin{align*} 
\supp_G(M)
	&=\{\fp\in\Proj H^*(G,k)\mid \text{$\gam_\fp k\otimes_k M$ is not projective} \}\\
\intertext{and its \emph{cohomological cosupport}} 
\cosupp_G(M)
	&=\{\fp\in\Proj H^*(G,k)\mid \text{$\Hom_k(\gam_\fp k,M)$ is not projective}\}.
\end{align*}

The result below reconciles these notions with the corresponding notions defined in terms of $\pi$-points. 

\begin{theorem}
  \label{th:supp-pisupp} 
Let $G$ be a finite group and $M$ a $kG$-module. Then
\[
\cosupp_{G}(M)=\picosupp_{G}(M) \quad\text{and}\quad
\supp_{G}(M)=\pisupp_{G}(M)\,,
\] 
regarded as subsets of $\Proj H^*(G,k)$.
\end{theorem}

\begin{proof} 
We use the fact that $\pisupp_{G}(\gam_\fp k) = \{\fp\}$; see  \cite[Proposition~6.6]{Friedlander/Pevtsova:2007a}. 
Then using Theorems~\ref{th:finite-groups} and~\ref{th:tensor-and-hom-pi}\,(ii)  one gets
\begin{align*} 
\fp\in\cosupp_{G}(M) 
	&\stackrel{\rm def}{\Iff} \Hom_k(\gam_\fp k,M) \text{ is not projective}\\ 
	&\stackrel{\ref{th:finite-groups}}{\Iff} \picosupp_{G}(\Hom_k(\gam_\fp k,M))\ne \varnothing \\ 
	&\stackrel{\ref{th:tensor-and-hom-pi}}{\Iff} \pisupp_{G}(\gam_\fp k) \cap \picosupp_{G}(M) \ne \varnothing \\ 
	&\Iff \fp\in\picosupp_{G}(M).
\end{align*} 
This gives the equality involving cosupports.

In the same vein, using Theorems~\ref{th:finite-groups} and~\ref{th:tensor-and-hom-pi}\,(i) one gets
\begin{align*} 
\fp\in\supp_{G}(M) 
	&\stackrel{\rm def}{\Iff} \gam_\fp k \otimes_k M \text{ is not projective}\\ 
	&\stackrel{\ref{th:finite-groups}}{\Iff} \pisupp_{G}(\gam_\fp k \otimes_k M) \ne \varnothing \\ 
	&\stackrel{\ref{th:tensor-and-hom-pi}}{\Iff} \pisupp_{G}(\gam_\fp k) \cap \pisupp_{G}(M) \ne \varnothing \\ 
	& \Iff \fp \in \pisupp_{G}(M).
\end{align*} 
This gives the equality involving supports.
\end{proof}

Here is a first consequence of this result; we are unable to verify it directly, except for closed points in the $\pi$-support and
$\pi$-cosupport.

\begin{corollary}
\label{co:maximal-elements} 
For any $kG$-module $M$ the maximal elements, with respect to inclusion, in $\picosupp_{G}(M)$ and
$\pisupp_{G}(M)$ coincide.
\end{corollary}

\begin{proof} 
Given Theorem~\ref{th:supp-pisupp}, this follows from \cite[Theorem~4.13]{\bik:2012b}.
\end{proof}

We continue with two useful formulas for computing cohomological
supports and cosupports; they are known from previous work
\cite{Benson/Carlson/Rickard:1996a, \bik:2012b} and are now accessible
from the perspective of $\pi$-points.

\begin{corollary}
\label{co:tensor-and-hom}
For all $kG$-modules $M$ and $N$ there are equalities
\begin{enumerate}
\item[\rm (i)] $\supp_{G}(M \otimes_k N) = \supp_{G}(M) \cap
\supp_{G}(N)$.
\item[\rm (ii)] $\cosupp_{G}(\Hom_k(M,N))= \supp_{G}(M) \cap
\cosupp_{G}(N)$.
\end{enumerate}
\end{corollary}

\begin{proof} This follows from Theorems~\ref{th:tensor-and-hom-pi}
and \ref{th:supp-pisupp}.
\end{proof}

We wrap up this section with a couple of examples. The first one 
shows that the $\pi$-support of a module $M$ may be properly contained
in that of $\End_k(M)$.

\begin{example}
\label{ex:klein3} 
Let $V$ be the Klein four group $\bbZ/2 \times \bbZ/2$ and $k$ a field 
of characteristic two. Thus, $\Proj H^*(V,k)=\bbP^{1}_{k}$, and a realisation 
of points of $\bbP^{1}_{k}$ as $\pi$-points of $kV$ was given in Example~\ref{ex:klein}. 
Let $M$ be the infinite dimensional $kV$-module described in Example~\ref{ex:klein2}. 
As noted there, the $\pi$-support of $M$ consists of a single point, namely, the closed 
point $[0,1]$. We claim
\[ 
\pisupp_{V}(\End_{k}(M)) = \{[0,1]\}\cup\{\text{generic point of $\bbP^{1}_{k}$}\}\,.
\] 
Indeed, since the $\pi$-cosupport of $M$ contains $[0,1]$, by Example~\ref{ex:klein2}, 
it follows from Theorem~\ref{th:tensor-and-hom-pi} that the $\pi$-cosupport of $\End_{k}(M)$ 
is exactly $\{[0,1]\}$. Corollary~\ref{co:maximal-elements} then implies that $[0,1]$ is 
the only closed point in the $\pi$-support of $\End_{k}(M)$. It remains to verify that the 
latter contains also the generic point.

Let $K$ be the field of rational functions in a variable $s$. The
$\pi$-point defined by $K[t]/(t^{p})\to K[x,y]/(x^{2},y^{2})$ with
$t\mapsto x+sy$ corresponds to the generic point of $\bbP^{1}_{k}$;
see Example~\ref{ex:klein}. The desired result follows once we verify
that the element $1 \otimes \id^{M}$ of $\End_{k}(M)_{K}$ is in the
kernel of $x+sy$ but not in its image. It is in the kernel because
$\id^{M}$ is a $kV$-module homomorphism. Suppose there exists an $f$
in $\End_{k}(M)_{K}$ with $(x+sy)f =1\otimes \id^{M}$. Then for each
$n\ge 0$, the identity $((x+sy)f)(u_n)=u_n$ yields
\[ 
f(v_{n}) + sf(v_{n-1}) = u_{n} + (x+sy)f(u_{n})\,.
\] 
Noting that $v_{-1}=0$, by convention, it follows that
\[ 
f(v_{n}) \equiv u_{n} + su_{n-1} + \cdots + s^{n}u_{0}
\] 
modulo the submodule $K(v_{0},v_{1},\cdots)$ of $M_{K}$. This cannot be, 
as $f$ is in $\End_{k}(M)_{K}$.
\end{example}

In Example~\ref{ex:klein2} it is proved that, for $M$ as above,
$\pisupp_V(M)\neq\picosupp_V(M)$. The remark below is a conceptual
explanation of this phenomenon, since $M$ is of the form $T(I_\fp)$
for $\fp=[0,1]$ in $\bbP^{1}_{k}$.

\begin{example}
\label{ex:T(I)} Given $\fp\in\Proj H^*(G,k)$, there is a $kG$-module
$T(I_\fp)$ which is defined in terms of the following natural
isomorphism
\[ 
\Hom_{H^*(G,k)}(\widehat H^*(G,-),I_\fp)\cong\sHom_{kG}(-,T(I_\fp))
\] 
where $I_\fp$ denotes the injective envelope of $H^*(G,k)/\fp$, $\widehat H^*(G,-)$ is Tate cohomology, and $\sHom_{kG}(-,-)$ is the set of 
homomorphisms in $\StMod kG$; see \cite[\S3]{Benson/Krause:2002a}.  The cohomological support and cosupport of this module have been computed in \cite[Lemma~11.10]{\bik:2011b} and \cite[Proposition~5.4]{\bik:2012b}, respectively. Combining this with Theorem~\ref{th:supp-pisupp} gives
\[ 
\pisupp_G(T(I_\fp))=\{\fp\}\quad\text{and}\quad
\picosupp_G(T(I_\fp))=\{\fq\in\Proj H^*(G,k)\mid\fq\subseteq\fp\}\,.
\]
\end{example}

\section{Stratification}
\label{se:stratification}
The results of this section concern the triangulated category structure of the stable module category, $\StMod(kG)$. Recall that a full subcategory $\sfC$ of $\StMod(kG)$ is \emph{localising} if it is a triangulated subcategory and is closed under arbitrary direct sums.
In a different vein, $\sfC$ is  \emph{tensor ideal} if for all $C$ in $\sfC$ and arbitrary $M$, the $kG$-module $C\otimes_{k}M$ is in $\sfC$. 

Following \cite[\S3]{\bik:2011b}, we say that the stable module category $\StMod(kG)$ is \emph{stratified} by $H^*(G,k)$ if for each
homogeneous prime ideal $\fp$ of $H^*(G,k)$ that is different from the maximal ideal of positive degree elements the localising subcategory
\[
\{M\in\StMod(kG)\mid\supp_G(M)\subseteq\{\fp\}\}
\]
admits no proper non-zero tensor ideal localising subcategory.

We are now in the position to give a simplified proof of \cite[Theorem~10.3]{Benson/Iyengar/Krause:2011b}. We refer the reader to \cite[Introduction]{\bik:2011b} for a version of this result dealing entirely with the (abelian) category of $kG$-modules.

\begin{theorem}
\label{th:stratification}
  Let $k$ be a field and $G$ a finite group. Then the stable module
  category $\StMod(kG)$ is stratified as a tensor triangulated
  category by the natural action of the cohomology ring $H^*(G,k)$.
  Therefore the assignment
  \begin{equation}\label{eq:supp}
    \sfC\longmapsto  \bigcup_{M\in\sfC}\supp_G(M)
\end{equation} 
induces a one to one correspondence between the tensor ideal localising subcategories of $\StMod(kG)$ and the subsets of $\Proj H^*(G, k)$.
\end{theorem}

\begin{proof}
It suffices to show that $\sHom_{kG}(M\otimes_k -,N)\neq  0$ whenever $M,N$ are $kG$-modules with $\supp_G(M)=\{\fp\}=\supp_G(N)$; see \cite[Lemma~3.9]{\bik:2011b}. By adjointness, this is equivalent to $\Hom_k(M,N)$ being non-projective. Thus the assertion follows from Corollary~\ref{co:tensor-and-hom}, once we observe that $\supp_G(N)=\{\fp\}$ implies $\fp\in\cosupp_G(N)$. But this is again a consequence of Corollary~\ref{co:tensor-and-hom}, since $\End_k(N)$ is non-projective by Lemma~\ref{le:end-projectivity}.

The second part of the assertion is a formal consequence of the first; see \cite[Theorem~3.8]{Benson/Iyengar/Krause:2011b}. The inverse map sends a subset $\mcV$ of $\Proj H^*(G, k)$ to the subcategory consisting of all $kG$-modules $M$ such that $\supp_G(M)\subseteq\mcV$.
\end{proof}

The next results concern $\stmod(kG)$, the full subcategory of  $\StMod G$ consisting of finite dimensional modules. A \emph{tensor ideal thick subcategory} $\sfC$ of $\stmod(kG)$ is a triangulated subcategory that is closed under direct summands and has the property that for any $C$ in $\sfC$ and finite dimensional $kG$-module $M$, the $kG$-module $C\otimes_{k}M$ is in $\sfC$.  The classification of the tensor ideal thick subcategories of $\stmod (kG)$ is the main result of \cite{Benson/Carlson/Rickard:1997a} and can be deduced from the classification of the tensor ideal localising subcategories of $\StMod (kG)$. This is based on the following lemma.

\begin{lemma}
\label{le:thick}
Let $M$ be a finite dimensional $kG$-module. Then $\supp_G(M)$ is a Zariski-closed subset of $\Proj H^*(G,k)$. Conversely, each Zariski-closed subset of $\Proj H^*(G,k)$ is of this form.
\end{lemma}

\begin{proof}
The first statement follows from \cite[Theorem~5.5]{\bik:2008a} and the second from \cite[Lemma~2.6]{\bik:2011a}.
\end{proof}

\begin{theorem}
\label{th:thick}
Let $G$ be a finite group and $k$ a field.  Then the assignment \eqref{eq:supp} induces a one to one correspondence between the tensor ideal thick subcategories of $\stmod(kG)$ and the specialisation closed subsets of $\Proj H^*(G,k)$.
\end{theorem}

\begin{proof}
Let $\sigma$ be the assignment \eqref{eq:supp}.  By Lemma~\ref{le:thick}, when $\sfC$ is a tensor ideal thick subcategory of $\stmod(kG)$, the  subset $\sigma(\sfC)$ of $\Proj H^{*}(G,k)$ is a union of Zarksi-closed subsets, and hence specialisation closed. Thus $\sigma$ restricted to $\stmod G$ has the desired image. Let $\tau$ be the map from specialisation closed subsets of $\Proj H^{*}(G,k)$ to $\stmod G$ that assigns $\mcV$ to the subcategory with objects
\[
\{M\in \stmod G\mid \supp_{G}M\subseteq \mcV\}\,.
\]
This is readily seen to be a tensor ideal thick subcategory. The claim is that $\sigma$ and $\tau$ are inverses of each other.

Indeed for any  $\mcV\subseteq \Proj H^{*}(G,k)$ there is an inclusion $\sigma\tau(\mcV)\subseteq \mcV$; equality holds if $\mcV$ is closed, by Lemma~\ref{le:thick}, and hence also if $\mcV$ is specialisation closed.

Fix a tensor ideal thick subcategory $\sfC$ of $\stmod G$. Evidently, there is an inclusion $\sfC\subseteq \tau\sigma(\sfC)$. To prove that equality holds, it suffices to prove that if $M$ is a finite dimensional $G$-module with $\supp_{G}(M)\subseteq \sigma(\sfC)$, then $M$ is in $\sfC$. Let $\sfC'$ be the tensor ideal localising subcategory of $\StMod (kG)$ generated by $\sfC$. From the properties of support, it is easy to verify that $\sigma(\sfC')=\sigma(\sfC)$, and then Theorem~\ref{th:stratification} implies $M$ is in $\sfC'$. Since $M$ is compact when viewed as an object in $\StMod G$, it follows by an argument analogous to the proof of \cite[Lemma~2.2]{Neeman:1992a} that $M$ is in $\sfC$, as desired.
\end{proof}

\begin{ack} 
Part of this article is based on work supported by the National Science Foundation under Grant No.\,0932078000, while DB, SBI, and HK were in residence at the Mathematical Sciences Research Institute in Berkeley, California, during the 2012--2013 Special Year in Commutative Algebra. The authors thank the Centre de Recerca Matem\`atica, Barcelona, for hospitality during a visit in April 2015 that turned out to be productive and pleasant. SBI and JP were partly supported by NSF grants DMS-1503044 and DMS-0953011,  respectively. We are grateful to Eric Friedlander for comments on an earlier version of this paper.
\end{ack}

\bibliographystyle{amsplain}
\newcommand{\noopsort}[1]{}
\providecommand{\MR}{\relax\ifhmode\unskip\space\fi MR }
\providecommand{\MRhref}[2]{%
  \href{http://www.ams.org/mathscinet-getitem?mr=#1}{#2}}
\providecommand{\href}[2]{#2}

\end{document}